\documentclass[12pt]{article}
\usepackage[utf8]{inputenc}
\usepackage[margin=3cm]{geometry}
\usepackage{lmodern}
\usepackage{fancyvrb}
\usepackage{graphicx}
\usepackage{amsthm,amsmath,amssymb}

\usepackage[colorlinks=true,citecolor=black,linkcolor=black,urlcolor=blue]{hyperref}

\theoremstyle{plain}
\newtheorem{theorem}{Theorem}
\newtheorem{lemma}[theorem]{Lemma}

\newtheorem{claim}[theorem]{Claim}

\theoremstyle{definition}
\newtheorem{defn}[theorem]{Definition}

\newtheorem{conj}[theorem]{Conjecture}

\theoremstyle{remark}
\newtheorem{remark}[theorem]{Remark}

\newcommand{\cH}{\ensuremath{\mathcal{H}}}%
\newcommand{\cC}{\ensuremath{\mathcal{C}}}%
\newcommand{\col}{\ensuremath{\mathrm{Col}}}%

\title{\bf On Ryser's conjecture for \texorpdfstring{$t$}{t}-intersecting and degree-bounded hypergraphs}

\author{Zolt\'an Kir\'aly \thanks{Research is supported by a grant (no.\ K
    109240) from the National Development Agency
    of Hungary, based on a source from the Research and Technology Innovation
    Fund.}\\
\small Department of Computer Science and \\[-0.8ex]
\small Egerv\'ary Research Group (MTA-ELTE)\\[-0.8ex]
\small E\"otv\"os University\\[-0.8ex]
\small P\'azm\'any P\'eter s\'et\'any 1/C, Budapest, Hungary.\\[-0.8ex]
\small Research was finished when the author was a visiting research fellow\\[-0.8ex]
\small at Alfréd Rényi Institute of Mathematics, Hungarian Academy of Sciences.\\
\small\tt kiraly@cs.elte.hu
\and
    Lilla T\'othm\'er\'esz \footnotemark[1]\\
\small Department of Computer Science and \\[-0.8ex]
\small Egerv\'ary Research Group (MTA-ELTE)\\[-0.8ex]
\small E\"otv\"os University\\[-0.8ex]
\small P\'azm\'any P\'eter s\'et\'any 1/C, Budapest, Hungary.\\[-0.8ex]
\small\tt tmlilla@cs.elte.hu
}

\begin{document}

\date{}
%\date{\dateline{}{}\\
%	\small Mathematics Subject Classifications: 05C15, 05C65, 05B40, 05B25}

\maketitle

{\let\thefootnote\relax\footnote{Mathematics Subject Classifications: 05C15, 05C65, 05B40, 05B25}}

\begin{abstract}
A famous conjecture (usually called Ryser's conjecture) that appeared in the  
 PhD thesis of his student, J.~R.~Henderson \cite{H}, states that
for an $r$-uniform
$r$-partite hypergraph $\cH$, the inequality
$\tau(\cH)\le(r-1)\!\cdot\! \nu(\cH)$ always holds. 

This conjecture is widely open, except in the case of
$r=2$, when it is
equivalent to K\H onig's theorem \cite{K}, and in the case of $r=3$, which was
proved by Aharoni in 2001 \cite{A}.

Here we study some special cases of Ryser's conjecture.
First of all, the most studied special case is when $\cH$ is
intersecting. Even for this special case, not too much is known: this
conjecture is proved only for $r\le 5$ in \cite{Gy,T2}. For $r>5$
it is also widely open.

Generalizing the conjecture for intersecting hypergraphs,  we conjecture
the following. 
If an $r$-uniform
$r$-partite hypergraph $\cH$ is $t$-intersecting (i.e., every two hyperedges
meet in at least $t<r$ vertices), then $\tau(\cH)\le r-t$. We prove this
conjecture for the case $t> r/4$.

Gy\'arf\'as
\cite{Gy} showed that Ryser's conjecture for intersecting hypergraphs is equivalent to
saying that the vertices of an $r$-edge-colored complete graph can be covered 
by $r-1$ monochromatic components.

Motivated by this formulation, we examine what fraction of the vertices can 
be covered by $r-1$ monochromatic components of \emph{different} colors in an
$r$-edge-colored complete graph.
We prove a sharp bound for this problem.

Finally we prove Ryser's conjecture for the very special case when the maximum
degree of the hypergraph is two.

  % keywords are optional
%  \bigskip\noindent \textbf{Keywords:} graph reconstruction
%  conjecture; Broglington manifold; Pipletti's classification
\end{abstract}

\section{Introduction}

A hypergraph is a pair $\cH=(V,E)$ where $V$ is a finite set (vertices), and $E$ is a multiset of subsets of $V$ (hyperedges). %If $\mathcal{E}$ is a set, i.e., each two hyperedges are different sets, then we talk about simple hypergraphs.  
A hypergraph is $r$-partite if its vertex set has a partition to $r$ nonempty
classes such that no hyperedge contains two vertices from the same class. 
We refer to the partite classes simply as \emph{classes} (note that in some
papers they are called sides). 
A set is called \emph{multi-colored} if it intersects every class in at most
one vertex, i.e., in an $r$-partite hypergraph every hyperedge is
multi-colored.
A hypergraph is $r$-uniform if all of its hyperedges have cardinality $r$. A
hypergraph is $d$-regular if every vertex is contained in exactly $d$
hyperedges.  
A hypergraph is $t$-intersecting if every pair of hyperedges have at least $t$
common vertices.  Throughout the paper we assume $0<t<r$ when speaking about
$t$-intersecting $r$-uniform hypergraphs.
A hypergraph is intersecting if it is 1-intersecting.

Let us introduce some more standard notations.
For a hypergraph $\cH$ with vertex set $V=V(\cH)$ and
hyperedge set $E=E(\cH)$
\[\textrm{the vertex covering number is: } \tau(\cH)=\min\{|T| : T\subseteq V,\; T\cap f\neq\emptyset \;\; \forall f\in
E\},\] 
\[\textrm{the edge covering number is: } \varrho(\cH)=\min\{|F| : F\subseteq E, \; \; \bigcup F=V\},\]
\[\textrm{the matching number is: } \nu(\cH)=\max\{|F| : F\subseteq E , \;  f_1\cap f_2=\emptyset \; \; \forall
f_1\ne f_2\in F \},\] 
\[\textrm{the maximum degree is: } \Delta(\cH)=\max\{|F|: \; F\subseteq E, \; \bigcap F\ne \emptyset\},\]
\[\textrm{the independence number is: } \alpha(\cH)=\max\{|X|: X\subseteq V,  \; f\not\subseteq X \;\; \forall
f\in E\},\] 
\[\textrm{the strong independence number is: }\alpha'(\cH)=\max\{|X|: X\subseteq V,  \; |f\cap X|\leq 1\ \;\; \forall
f\in E\}.\]

A famous conjecture of Ryser (which appeared in the  
 PhD thesis of his student, J.R.~Henderson \cite{H}) states that
for an $r$-uniform
$r$-partite hypergraph $\cH$, we have 
$\tau(\cH)\le(r-1)\!\cdot\! \nu(\cH)$. 

This conjecture is widely open, except in the special case of
$r=2$, when it is
equivalent to K\H onig's theorem \cite{K}, and when
$r=3$, which was
proved by Aharoni in 2001 \cite{A}, using topological results from \cite{AH}. We
mention also some related results. Henderson \cite{H} showed that the conjecture
cannot be improved if $r-1$ is a prime power. 
Haxell and Scott \cite{HS} showed that the constant in the conjecture cannot be smaller than $r-4$ for all but finitely many values of $r$.
 F\"uredi \cite{F} proved that
the fractional covering number is always at most $(r-1)\!\cdot\!\nu(\cH)$, and
Lov\'asz \cite{L} proved that the fractional matching number is always at
least $\frac2r\!\cdot\!\tau(\cH)$.
The hypergraphs achieving $\tau(\cH) = (r-1)\!\cdot\! \nu(\cH)$ have also been investigated, but this problem is also widely open. Haxell, Narins and Szab\'o characterized the sharp examples for $r=3$ \cite{HNSz1,HNSz2}. For larger values of $r$, truncated projective planes give an infinite family of sharp examles. 
Apart from these, there are some sporadic examples \cite{ABW,AP,FHMW,MSY}, moreover, Abu-Khazneh, Bar\'at, Pokrovskiy and Szab\'o \cite{ABPSz} constructed another infinite family of extremal hypergraphs but projective planes play also an important role in their construction.

Here we study some special cases of Ryser's conjecture.
First of all, the most studied special case is when $\nu=1$, i.e., when $\cH$ is
intersecting. Even for this case, not too much is known.  Gy\'arf\'as
\cite{Gy} showed that this special case of the conjecture is equivalent to
saying that the vertices of an $r$-edge-colored complete graph can be covered 
by $r-1$ monochromatic components (see below).
He also proved this conjecture for $r\le 4$ \cite{Gy}, and later Tuza \cite{T2}
proved it for $r=5$.  For $r>5$ this conjecture is also widely open. Some
recent papers study this special case, e.g., see \cite{AB, HS, FHMW, MSY}.
For intersecting hypergraphs, we generalize Ryser's conjecture by conjecturing the following.
If an $r$-uniform
$r$-partite hypergraph $\cH$ is $t$-intersecting, then $\tau(\cH)\le r-t$. We
prove this conjecture for the case $r>t> r/4$. This question was also studied (independently) by Bustamante and Stein, see \cite{BS}.

The construction of Gy\'arf\'as \cite{Gy} (see also in \cite{KZ}) is the
following. We associate a multi edge-colored graph to an $r$-partite $r$-uniform
hypergraph. 

\begin{defn} \label{def:Gyarfas_constr}
For an $r$-partite $r$-uniform hypergraph $\cH$, let $G=G(\cH)$ be the
following multi edge-colored graph: 

The vertex set of $G$ is $V(G)=E(\cH)$.
Two vertices $u,v\in V(G)$ are connected by an edge if the corresponding hyperedges $E_u, E_v\in E(\cH)$ have a nonempty intersection. The edge $uv$ is colored by the colors $\{i: E_u\textrm{ and } E_v \textrm{ share a vertex from the $i^{th}$ class}\}$. 
We denote the set of colors of edge $uv$ by $\col(uv)$. If $i\in \col(uv)$, then we say that the edge $uv$ \emph{has} the color $i$.
\end{defn}

Note that if $\cH$ is intersecting,
then $G$ is a complete graph.

\begin{remark}
The original construction of Gy\'arf\'as colored each edge $uv$ by only one
color, chosen arbitrarily from $\col(uv)$. 
\end{remark}

\begin{remark}\label{remark_multi}
  The color sets we defined in this way are transitive: if
  $i\in  \col(uv)\cap \col(vw)$, then $i\in  \col(uw)$. We call a complete
  graph $G$ \emph{multi $r$-edge-colored} if for each distinct vertex pair $\{u,v\}$
  we have $\emptyset\ne \col(uv)\subseteq[r]=\{1,\dots,r\}$ and if the
  coloring is transitive. 
  In a multi $r$-edge-colored graph, a \emph{monochromatic component} of color
  $i$ is a component of the subgraph formed by the edges using the color $i$.
  Note that -- as the coloring is transitive -- if $U$ is the vertex set
  of a monochromatic component of color $i$, then for every $u\ne v\in U$ we
  have $i\in\col(uv)$; in other words, the edges of $G$ having color $i$ form
  a partition of $V(G)$ into $i$-colored cliques.
Each vertex of $\cH$ in the class $i$ corresponds to
one maximal clique of color $i$, which is also a monochromatic ($i$-colored)
connected component of $G$. A set of vertices $T\subseteq V(\cH)$ covers the
hyperedges of $\cH$  (as in the definition of $\tau$) if and only if the
monochromatic components corresponding to its elements cover $V(G)$. 
\end{remark}

\begin{remark}\label{remark_multi2}
  We also note that for any edge-colored complete graph we can consider the
  color-transitive closure: for any edge $uv$ we define
  $\col(uv)=\{i \;|\; u \textrm{ and } v$
  are in the same monochromatic component of color $i\}$.
  The vertex sets of the monochromatic components of this multi edge-colored graph are the same as the vertex sets of the monochromatic components of the original edge-colored graph.
\end{remark}

Ryser's conjecture for intersecting hypergraphs is equivalent to the statement
that $r-1$ monochromatic components can cover $V(G(\cH))$.  The more general
conjecture for $t$-intersecting hypergraphs is equivalent to the statement
that for every multi $r$-edge-colored complete graph, where each edge has at
least $t$ colors, there is a set of $r-t$ monochromatic components that cover
the vertices (if $t<r$).

For the case of $r$-edge-colored complete graphs, we also study the following
problem: 
What fraction of the vertices can be
covered by $r-1$ monochromatic components of \emph{different} colors?
We prove a sharp bound for this problem, namely
$\big(1-\frac{r-2}{(r-1)^2}\big)\cdot |V(G)|$.
In the hypergraph language, this corresponds to the question of ``How many
hyperedges can be covered by a multi-colored set of size $r-1$ in an
intersecting $r$-partite $r$-uniform hypergraph?'' We show that the
hypergraphs giving the minimum are exactly the hypergraphs that can be obtained from a truncated
projective plane by replacing each hyperedge by $b$ parallel copies for some
integer $b$.  

Finally we prove Ryser's conjecture for the very special case when the maximum
degree of the hypergraph is two, i.e., when no vertex is contained in three or
more hyperedges. 

The preliminary version of this paper can be found in \cite{KTtech}.

\section{The \texorpdfstring{$t$}{t}-intersecting case}

\begin{conj}\label{conj:t-intersecting}
Let $\cH$ be an $r$-uniform $r$-partite $t$-intersecting hypergraph with $1\leq t \leq r-1$.
Then $\tau(\cH) \leq r-t$.
\end{conj}

\begin{theorem}\label{p:t-intersecting}
If $\cH$ is an $r$-uniform $r$-partite $t$-intersecting hypergraph and
$\frac{r}{4}< t \leq r-1$, then $\tau(\cH) \leq r-t$.
\end{theorem}

Using Gy\'arf\'as' construction (Definition \ref{def:Gyarfas_constr}),
Theorem \ref{p:t-intersecting} follows from 
the following statement: 

\begin{theorem}\label{thm:t-metszo_eros}
  Let $G$ be a multi $r$-edge-colored complete graph
  where each edge has at least $t$ different colors. 
If $r-1\geq t> \frac{r}{4}$, then $V(G)$ can be covered by at most $r-t$
monochromatic components. 
\end{theorem}

\begin{remark}\label{not_stronger}
  Conjecture \ref{conj:t-intersecting} is seemingly a
  strengthening of Ryser's conjecture for
intersecting hypergraphs (which corresponds to $t=1$).
However, the statement is stronger for smaller $t$
values.

To see this, suppose the conjecture is proved for a fixed $t$ and for every
$r>t$. 
To prove it for $t+1$, suppose we are given a multi $r$-edge-colored complete graph
where each edge has at least $t+1$ different colors and $r>t+1$.
Deleting color $r$ from every $\col(uv)$, we get a multi
$(r\!-\!1)$-edge-colored complete graph where each edge has at least
$t$ different colors, so by the assumption its vertex set can be covered by
$r-1-t=r-(t+1)$ monochromatic components.
\end{remark}

\begin{remark}\label{rem_BS}
  In a recent manuscript \cite{BS},
  Bustamante and Stein formulated independently the same 
conjecture (we are thankful for the reviewer who raised our attention to it).
They proved that the conjecture is true if $r-1\geq t\ge \frac{r-2}{2}$.
We note that our theorem is stronger (except that their result contains the
well-known case $r=4,\; t=1$ while our result does not).
\end{remark}

%% We can suppose that each mo\-no\-chromatic component is a clique. Indeed, if a mo\-no\-chro\-matic component is not a clique, we can add its color to the color set of those edges that connect two vertices from the same monochromatic component but do not have this color. This way the condition that each edge has at least $t$ colors is not violated. Also, the set of vertices of each monochromatic component remains the same. Hence throughout the proof we suppose that the monochromatic components are cliques.

\begin{proof}[Proof of Theorem \ref{thm:t-metszo_eros}]
We assume $[r]=\{1,2,\ldots,r\}$ is the set of colors, and if $x$ is a vertex and $I\subseteq [r]$ is a set of colors, then we denote by $\cC(x,I)$ the set of monochromatic components containing $x$ and having a color in $I$. 

The proof goes by induction on the number vertices. If $|V(G)|\le 2$, we can
cover $V(G)$ by 1 monochromatic component.
If there are $x\ne y\in V(G)$ where
$|\col(xy)|=r$, then contract the edge $xy$ to a vertex $x^*$
(by color-transitivity, for any vertex
$z\ne x,y$ we have $\col(zx)=\col(zy)$, so we define $\col(zx^*)=\col(zx)$).
By induction, the graph obtained can be covered by at most $r-t$ monochromatic
components. It is easy to see that the preimages of these components are
monochromatic and cover $V(G)$. So from this point we may (and will)
suppose that $|\col(xy)|<r$ for every pair $x\ne y$.

First we prove some special cases.

\begin{lemma}\label{lem:t-metszo_eros}
Let $G$ be a multi $r$-edge-colored complete graph, where
each edge has
exactly $t$ different colors.
If $t+1\le r\le 4t-2$, then $V(G)$ can be covered by at most $r-t$
monochromatic components. 
\end{lemma}

\begin{proof}
Take any edge $xy$. Without loss of generality, we can suppose that
$\col(xy)=I=[t]$.

First consider the case $r\leq 2t$. 
Let $J=[r-t]$. Now $J\subseteq I$ since $r-t\leq t$.
We claim that $\cC(x,J)=\cC(y,J)$ covers $V(G)$. If a vertex $z$ is not
covered, then $\col(xz)=\col(yz)=\{r-t+1,\ldots,r\}$. However, since each
monochromatic component is a clique, we get $\{r-t+1,\ldots,r\}\subseteq \col(xy)=I$, so
$t=r$ contradicting 
the assumption $t<r$. 

Thus it remains to prove the case $r>2t$. Let
$j=\lfloor\frac{r}{2}\rfloor-t$ and  
$J=\{t+1, \dots, t+j\}$ if $j>0$, and
$J=\emptyset$ otherwise. Take $\cC(x,I)\cup \cC(x,J)\cup \cC(y,J)$. 
We claim that these $t+2j\le r-t$ monochromatic components cover the
vertices of $G$. 
If a vertex $z$ is not
covered, then $\col(xz)\subseteq \{t+j+1, \dots, r\}$ and 
$\col(yz)\subseteq \{t+j+1, \dots, r\}$ and, as the coloring is transitive,
$\col(xz)\cap \col(yz)\subseteq I$,
thus $\col(xz)\cap \col(yz)=\emptyset$.
However, $|\col(xz)|=|\col(yz)|=t$, so $2t\leq r-t-j$, i.e., 
$2t\leq \lceil\frac{r}{2}\rceil$ or equivalently $r\geq 4t - 1$, a contradiction.
\end{proof}

\begin{lemma}
  Let $G$ be a multi $r$-edge-colored complete graph where each
edge has 
at least $t$ different colors.
If $t+1\le r\leq 4t-1$ and there is an edge $xy$ with $t<|\col(xy)|<r$, 
then $V(G)$ can be covered by at most $r-t$
monochromatic components. 
\end{lemma}

\begin{proof}
Take an edge $xy$ with $|\col(xy)|>t$. Without loss of generality, we can suppose that
$\col(xy)=I=[\ell]$ where $t<\ell<r$.

First consider the case $r\leq t+\ell$. 
Let $J=[r-t]$. Now $J\subseteq I$.
We claim that $\cC(x,J)=\cC(y,J)$ covers $V(G)$. If a vertex $z$ is not
covered, then $\col(xz)=\col(yz)=\{r-t+1,\ldots,r\}$. However, since the
coloring is transitive, we get $\{r-t+1,\ldots,r\}\subseteq I$, so
$\ell=|I|=r$ contradicting 
the assumption $\ell<r$. 

Thus it remains to prove the case $r>t+\ell$. Let
$j=\lfloor\frac{r-t-\ell}{2}\rfloor$ and  
$J=\{\ell+1, \dots, \ell+j\}$ if $j>0$, and
$J=\emptyset$ otherwise. Take $\cC(x,I)\cup \cC(x,J)\cup \cC(y,J)$. 
We claim that these $\ell+2j\le r-t$ monochromatic components cover the
vertices of $G$. 
If a vertex $z$ is not
covered, then $\col(xz)\subseteq \{\ell+j+1, \dots, r\}$ and 
$\col(yz)\subseteq \{\ell+j+1, \dots, r\}$ and, as each
monochromatic component is a clique, $\col(xz)\cap \col(yz)=\emptyset$.
However, $|\col(xz)|\geq t$ and $|\col(yz)|\geq t$, so $2t\leq r-\ell-j$, i.e., 
$t\leq \lceil\frac{r-t-\ell}{2}\rceil$ or equivalently 
$r\geq 3t+\ell-1\geq 4t$, a contradiction.
\end{proof}

It remains to prove the case $r=4t-1$ and $|\col(xy)|=t$ for each $x\ne y$.
Let $k$ be the largest integer $j$ such that there is a triangle in $G$ with
$j$ colors occurring on all three edges. 
Let $xyz$ be a triangle with $k$ common colors on its edges.
Let us introduce some further notations.

Let $K=\col(xy)\cap\col(yz)\cap\col(zx)$
and $X=\col(yz) - K$
and $Y=\col(xz) - K$
and $Z=\col(xy) - K$, finally let 
$S=[r]- (K\cup X\cup Y\cup Z)$.

\noindent
For a set $A$, $A'$ always denotes a subset of $A$. Moreover, we denote $A- A'$ by $A''$. Note that $|K|=k$ and $|X|=|Y|=|Z|=t-k$.

\bigskip

\par \textbf{Case 0: } $k=0$.

Now $|V(G)|\le r+1$ as no two incident edges may have the same color. 
Let $V(G)=\{u_1,\ldots,u_n\}$, where $n\le r+1$, and let $c_i\in \col(u_{2i-1}u_{2i})$ for $1\le i\le n/2$.
If $n$ is even, then consider $\cC(u_1,c_1)\cup\cC(u_3,c_2)\cup\ldots\cup\cC(u_{n-1},c_{n/2})$, otherwise  consider $\cC(u_1,c_1)\cup\cC(u_3,c_2)\cup\ldots\cup\cC(u_{n-2},c_{(n-1)/2})\cup\cC(u_n,c)$, where $c$ is an arbitrary color of an edge incident to $u_n$. These (at most $\lfloor(n+1)/2\rfloor\le \lfloor(r+2)/2\rfloor$) monochromatic components obviously cover $V(G)$, and $\lfloor(r+2)/2\rfloor\le r-t$ as $r=4t-1$. 

\bigskip

\par \textbf{Case 1: } $0<3k\le t$.

Choose $Y'\subseteq Y$ and $Z'\subseteq Z$ so that $|Y'|+|Z'|=t+k-1$. This is
possible, since $|Y|+|Z|= 2t-2k \geq t+k$ because $t\geq 3k$.

Take the following monochromatic components:
$\cC(x,K\cup Y\cup Z)\cup \cC(y,Y')\cup \cC(z,Z')$.
The number of components chosen is at most $(2t-k)+(t+k-1)=3t-1=r-t$.

\begin{claim}
	The components $\cC(x,K\cup Y\cup Z)\cup \cC(y,Y')\cup \cC(z,Z')$ cover each vertex.
\end{claim}
\begin{proof}
Suppose that a vertex $w$ is not covered. 
%Then we claim that $\col(xw)\subseteq X\cup S$ and $\col(yw)\subseteq X\cup Y''\cup S$ and $\col(zw)\subseteq X\cup Z''\cup S$.
Then $\bigl(\col(xw)\cup\col(yw)\cup\col(zw)\bigr)\cap K=\emptyset$.
Similarly, $\col(xw)\cap Y=\emptyset$ and $\col(yw)\cap Y'=\emptyset$.
 
We claim that also $\col(zw)\cap Y=\emptyset$. Indeed, as $Y\subseteq \col(xz)$, if $zw$ had a color from $Y$, then $xw$
would also have that color (since the coloring is transitive), a contradiction.

By the same reasoning, $\col(xw)\cap Z=\emptyset$, $\;\col(yw)\cap Z=\emptyset$ and $\col(zw)\cap Z'=\emptyset$.

As a consequence, $\col(xw)\subseteq X\cup S$, $\;\col(yw)\subseteq X\cup Y''\cup S$ and $\col(zw)\subseteq X\cup Z''\cup S$.

Next we claim that the colors in $X$ can occur altogether (counting with multiplicity) at most $t$ times on the edges $xw, yw$ and $zw$. 
Let $c\in X$ be a color. If it occurs more than once on edges $xw, yw$ and $zw$, then it is in $\col(yw)\cap\col(zw)$ but $c\not\in\col(xw)$. To see this, note that if $c\in\col(xw)\cap\col(yw)$, then $c\in\col(xy)$ contradicting $X\cap\col(xy)=\emptyset$; similarly $c\not\in\col(xw)\cap\col(zw)$. 
By the choice of $k$, $|\col(yw)\cap\col(zw)|\le k$. Hence the colors in $X$ occur at most $|X|+k\leq t$ times on the edges $xw, yw$ and $zw$.

Each color in $S$ can only occur once on $xw, yw$ and $zw$, since by color-transitivity, a color occurring on at least two of the edges $xw, yw$ and $zw$ would also occur on one of the edges $xy, yz$ and $zx$, and that would contradict the definition of $S$.

 Hence counting the colors of the edges $xw, yw$ and $zw$:
 $3t \le |S|+|Z''|+|Y''|+t=|S|+(|Y|+|Z|-(|Y'|+|Z'|))+t=(4t-1-(3t-2k))+(2t-2k-(t+k-1))+t=(t+2k-1)+(t-3k+1)+t=3t-k$,  which is a contradiction.
 \end{proof}

 \bigskip

 %\paragraph{Case 2: } $3k>t$.
 \par \textbf{Case 2: } $3k>t$.

 If $|X|+|Y|+|Z|= 3t-3k \geq 2k-1$, then choose $X'\subseteq X$, $Y'\subseteq Y$ and $Z'\subseteq Z$ so that $|X'|+|Y'|+|Z'|= 2k-1$.
 If $|X|+|Y|+|Z|= 3t-3k < 2k-1$, then let $X'=X$, $Y'=Y$ and $Z'=Z$.

 Take the following monochromatic components:
 $\cC(x,K\cup X'\cup Y\cup Z)\cup \cC(y,Y')\cup \cC(z,X\cup Z')$.
 The number of components chosen is at most 
 $|K|+|X|+|Y|+|Z|+|X'|+|Y'|+|Z'|\leq k+3(t-k)+2k-1=3t-1= r-t$.

 We claim that the components chosen cover each vertex.
 Suppose that there is a vertex $w$ which is not covered.
 Similarly to the previous case, it is easy to prove that the colors of $xw, yw$ and $zw$
 are all from $S\cup X''\cup Y''\cup Z''$, and each color is used at most once altogether on these three edges. Hence $3t\leq |S|+|X''|+|Y''|+|Z''|$.

 If $3t-3k\geq 2k-1$, then $3t\leq |S|+|X''|+|Y''|+|Z''|=4t-1-(|K|+|X'|+|Y'|+|Z'|)=4t-1-(k+2k-1)=4t-3k<3t$ since $t<3k$. This is a contradiction.

 If $3t-3k < 2k-1$, then $3t\leq |S\cup X''\cup Y''\cup Z''|=|S|=4t-1-(3t-2k)=t+2k-1$. But this implies $2t\leq  2k-1$, hence $k>t$, which  contradicts the assumption that each edge has exactly $t$ colors.
This finishes the proof of Theorem \ref{thm:t-metszo_eros}.
\end{proof}

\begin{remark}
  We think that with a more diversified case analysis, Theorem \ref{thm:t-metszo_eros} can be extended to the case $t\ge r/5$. Note however, that the case $t=r/6$ would include the first unsolved case of Ryser's conjecture for intersecting hypergraphs.
\end{remark}

\section{Covering large fraction by few monochromatic components}

In this section, we give a sharp bound for the ratio of vertices that can be 
covered by $r-1$ monochromatic components of pairwise different colors in an
$r$-edge colored complete graph. By Remark \ref{remark_multi2}, we can assume that the monochromatic components of the graph are cliques, since in
the color-transitive closure of a graph, the monochromatic components have the same vertex sets as in the original graph.

\begin{theorem}\label{thm:cover_diff}
  Let $G$ be a multi $r$-edge-colored complete graph on $n$ vertices.
  Then at least 
   $\big(1-\frac{r-2}{(r-1)^2}\big)\cdot n$ vertices of
   $G$ can be covered by
   $r-1$ monochromatic components of pairwise different colors, and this bound
   is sharp for infinitely many values of $r$. Moreover, the $r-1$ monochromatic components can be chosen so that their intersection is nonempty.
\end{theorem}

Applying the construction of Gy\'arf\'as (Definition
\ref{def:Gyarfas_constr}), we get the following equivalent statement for hypergraphs.

\begin{theorem}\label{thm:cover_diff_hyp}
   Let $\cH$ be an $r$-partite $r$-uniform intersecting hypergraph.
   Then at least 
   $\big(1-\frac{r-2}{(r-1)^2}\big)\cdot |E(\cH)|$ hyperedges of
   $\cH$ can be covered by a multi-colored set of size
   $r-1$, and this bound
   is sharp for infinitely many values of $r$. Moreover, the cover can be
   chosen so that it is a subset of some hyperedge of $\cH$.
\end{theorem}

The following strengthening of Ryser's conjecture was phrased by Aharoni et al. \cite[Conjecture 3.1]{ABW}: ``In an intersecting $r$-partite $r$-uniform hypergraph $\cH$, there exists a class of size $r-1$ or less, or a
cover of the form $e- \{x\}$ for some $e \in E$ and $x \in e$.''
This conjecture was disproved in \cite{FHMW}. Note however, that by Theorem
\ref{thm:cover_diff_hyp}, if we require the cover to be multi-colored,
then additionally requiring
it to be a subset of a hyperedge does not decrease the number of
coverable hyperedges in the worst case.

We call the reader's attention to the fact that, although our result is sharp
for infinitely many values of $r$, in all our examples showing sharpness every
class has exactly $r-1$ vertices, thus they are far from exhibiting a
counterexample to Ryser's conjecture. 

\begin{proof}[Proof of Theorem \ref{thm:cover_diff}]
  We call an edge-coloring of $G$ \emph{spanning} if for every color $c$ and vertex $u$ there is an edge $uv$ of $G$ such that $c\in\col(uv)$.
  If the edge-coloring of $G$ is not spanning, then we can cover all the vertices of $G$ by $r-1$ monochromatic components of pairwise different colors. Indeed, if there is a vertex $v$ and a color $i$ such that no edge incident to $v$ has color $i$, then $\cC(v,[r]-\{i\})$ covers the vertices of $G$.

  Now suppose that the coloring of $G$ is spanning. For $r=2$ we can cover the vertex set by one  monochromatic component
  by a well-known folklore observation, so we 
  may assume $r\ge 3$.
  Let the number of monochromatic components of color $i$ be $k_i$. Let us
  denote the set of monochromatic components of color $i$ by
  $\mathcal{C}_i$. We may suppose
  that $k_1\ge k_2\ge\ldots\ge k_r\ge 2$, otherwise (if $k_r=1$) we are
  done. In the following proof, we will think of monochromatic components as vertex sets, hence when we write $C\in \mathcal{C}_i$, we mean that $C$ is the vertex set of a monochromatic component of color $i$.

\par \textbf{Case 1: } $k_1\ge r-1$.
We have

\begin{equation} \label{eq:komp_kul}
\sum_{C\in \mathcal{C}_1,\; C'\in \mathcal{C}_r} |C-C'|=(k_r-1)\cdot n,
\end{equation}
since each vertex occurs in exactly one component of color $r$ and one
component of color 1. Hence each vertex is counted $k_r-1$ times for the
$k_r-1$ components of color $r$ that does not contain it. 

 From (\ref{eq:komp_kul}) it follows that among the $k_1\cdot k_r$ sets
$\{C-C': C\in \mathcal{C}_1,C'\in\mathcal{C}_r\}$,
there is one which has size at most
$\frac{k_r-1}{k_1\cdot  k_r}\cdot n$.
Let $C_1-C'_r$ be such a set with minimum cardinality. As $k_1\ge k_r$ we have
$\frac{k_r-1}{k_r}\le\frac{k_1-1}{k_1}$, so
$\frac{k_r-1}{k_1\cdot k_r}\cdot n\le \frac{k_1-1}{k_1^2}\cdot n$.
Using $2\le r-1\le k_1$ we  also have
$\frac{k_1-1}{k_1^2}\le \frac{r-2}{(r-1)^2}$, so
$\frac{k_r-1}{k_1\cdot  k_r}\cdot n\le \frac{r-2}{(r-1)^2}\cdot n$.

We claim that $C_1\cap C'_r\ne\emptyset$. Indeed, take a vertex $x\in C_1$. If
$C_1\cap C'_r=\emptyset$, then $|C_1 - \cC(x,\{r\})|<|C_1|=|C_1 - C'_r|$ which
contradicts the minimality of $C_1-C'_r$. Thus we can choose a vertex $x$ in
$C_1\cap C'_r$. Take $\cC(x,[r]-\{1\})$.  These components cover each
vertex outside $C_1-C'_r$, hence at least $(1-\frac{r-2}{(r-1)^2})\cdot n$
vertices.

%\paragraph{Case 2:} $k_1\le r-1$ (i.e., $k_i\le r-1$ for all $i$).
\par \textbf{Case 2: } $k_1\le r-1$ (i.e., $k_i\le r-1$ for all $i$).

Notice that Case 1 and Case 2 overlap. However, this overlapping categorization will be convenient when examining sharpness.

For a vertex $v$ and a color $i\in [r]$, let $d_i(v)=|\{u\in V(G): \col(uv)=\{i\}\}|$, i.e., the number of neighbors of $v$ that are connected to $v$ by an edge having only color $i$.
It is enough to show that there exists $v\in V$ and $i \in \{1,\dots , r\}$ such that 
$d_i(v)\leq \frac{r-2}{(r-1)^2}\cdot n$.
Indeed, in this case $\cC(v,[r]-\{i\})$ cover each vertex except those that are connected to $v$ by an edge of 
unique color $i$, that is, at most $\frac{r-2}{(r-1)^2}\cdot n$ vertices are uncovered.

Let $m_i=|\{uv\in E(G): \col(uv)=\{i\}\}|$, %the number of edges whose unique color is $i$, 
and $M_i=|\{uv\in E(G): i\in \col(uv)\}|$. %the number of edges having color $i$. 
Since 
$\sum_{v\in V}d_i(v)=2m_i$, it is enough to show that there exists a color 
$i$ such that $m_i\leq \frac{r-2}{2(r-1)^2}\cdot n^2$.
For this, it is enough to show that $\sum_{i=1}^r m_i\leq \frac{r(r-2)}{2(r-1)^2}\cdot n^2$.
We have $\sum_{i=1}^r m_i = \binom{n}{2} -t$ where $t$ denotes the number of edges having multiple colors.

It is not hard to see that
$t\geq \frac{1}{r-1}\cdot \big[\sum_{i=1}^r M_i-\binom{n}{2}\big]$, since each
edge has at most $r$ colors. 

\begin{claim} \label{cl:M_i bound}
  If $\ell=k_i\le r-1$, then $M_i\geq \frac{n^2}{2\ell}-\frac{n}{2}
  \geq \frac{n^2}{2(r-1)}-\frac{n}{2}$.
\end{claim}
\begin{proof}
Let the cardinalities of the components of color $i$ be $\gamma_1,\dots,
\gamma_\ell$. 
Then $M_i=\binom{\gamma_1}{2} + \dots + \binom{\gamma_\ell}{2}=
\frac{\gamma_1^2+ \dots +\gamma_\ell^2}{2}-\frac{\gamma_1+ \dots +\gamma_\ell}{2}=
\frac{\gamma_1^2+ \dots +\gamma_\ell^2}{2}-\frac{n}{2}$.

Now it is enough to show that $\frac{\gamma_1^2+ \dots +\gamma_\ell^2}{2}\geq
\frac{n^2}{2\ell}$ 
but this follows from the Arithmetic Mean--Quadratic Mean Inequality.
\end{proof}

Using the claim, we get that
$t\geq \frac{1}{r-1}\cdot \big[\sum_{i=1}^r M_i-\binom{n}{
    2}\big]\geq \frac{1}{r-1}\cdot \big[\frac{r(n^2-(r-1)n)}{2(r-1)}-\binom{n}{2}\big]=\frac{rn^2-r(r-1)n-(r-1)n^2+(r-1)n}{2(r-1)^2}=
\frac{n^2}{2(r-1)^2}-\frac{n}{2}$.

So $\sum_{i=1}^r m_i= \binom{n}{2}-t\leq \binom{n}{2}-\frac{n^2}{2(r-1)^2}+\frac{n}{2}=
\frac{(r-1)^2 n^2-(r-1)^2 n-n^2+(r-1)^2 n}{2(r-1)^2}=
\frac{r(r-2)n^2}{2(r-1)^2}$.

For the proof of sharpness see Theorem \ref{thm:cover_diff_sharp}.
\end{proof}

\subsection{Characterization of sharp examples}

In this subsection we characterize the sharp examples for Theorem \ref{thm:cover_diff}. 
For this, we will need the definition of an affine plane of order $r-1$.

\begin{defn}\label{def:aff_plane}
	An incidence structure $\mathcal{A}=(\mathcal{P},\mathcal{L})$, where the elements of $\mathcal{P}$ are referred to as the points, and the elements of $\mathcal{L}$ are referred to as the lines is called an \emph{affine plane of order $r-1$} if the following five conditions hold. 
\begin{description}
	\item[(i)] Every pair of points are connected by exactly one line.
	\item[(ii)] For each point $x$ and line $L$ such that $x\notin L$, there exists exactly one line $L'$ such that $x\in L'$, but $L'$ is disjoint from $L$.
	\item[(iii)]  Each line contains at least 2 points.
	\item[(iv)] Each point is incident with at least 3 lines.
	\item[(v)] The maximum number of pairwise parallel lines is $r-1$.
\end{description}
\end{defn}

We also need the following definition.
\begin{defn}
We call a multi edge-colored 
complete graph $G$ the \emph{blowup of an affine plane} if there 
exists an affine plane $\mathcal{A}=(\mathcal{P},\mathcal{L})$, a positive integer $b$ and a function $f:V(G)\rightarrow \mathcal{P}$ such that
\begin{itemize}
	\item the lines of $\mathcal{A}$ are colored such that two lines have the same color if and only if they are disjoint (i.e., parallel),
	\item for each point $p\in \mathcal{P}$, $|\{v\in V(G): f(v)=p\}|=b$
	\item $i\in \col(uv)$ if and only if $f(u)$ and $f(v)$ are incident to a common line of color $i$ (note that this includes the case if $f(u)=f(v)$).
\end{itemize}
 %whose lines are colored such that two lines have the same color if and only if they are disjoint (i.e., parallel),
 %and a positive integer $b$, such that to every point $p\in \mathcal{P}$ of the affine plane, $b$ vertices correspond in $V(G)$, and two vertices are connected by an edge having color $i$ if and only if the corresponding points in $\mathcal{A}$ are incident to a common line of color $i$ (this includes also the case if the two points correspond to the same point of $\mathcal{A}$).
\end{defn} 

\begin{theorem}\label{thm:cover_diff_sharp}
For a multi $r$-edge-colored complete graph $G$ on $n$ vertices, the maximum number of vertices coverable by $r-1$ monochromatic components of pairwise different colors equals $\big(1-\frac{r-2}{(r-1)^2}\big)\cdot n$ if and only if $G$ is a blowup of an affine plane.
\end{theorem}

\begin{proof}
Suppose $G$ is a sharp example, i.e., no $r-1$ monochromatic components of
pairwise different colors can cover more than
$\big(1-\frac{r-2}{(r-1)^2}\big)\cdot n$ vertices and
$\big(1-\frac{r-2}{(r-1)^2}\big)\cdot n$ is an integer.

As noted in the beginning of the proof of Theorem \ref{thm:cover_diff}, if the edge-coloring of $G$ is not spanning or $r=2$, then all the vertices of $G$ can be colored by $r-1$ monochromatic components of pairwise different colors, hence in these cases, there is no sharp example.

Now suppose that the coloring of $G$ is spanning, and $r\geq 3$. We examine the proof of Theorem \ref{thm:cover_diff} to see how the inequalities can be equalities.
In Case 1, $k_1=\dots = k_r=r-1$ for a sharp example, since otherwise $\frac{k_r-1}{k_1\cdot k_r}\cdot n$ would be strictly smaller than $\frac{r-2}{(r-1)^2}\cdot n$.

Also in Case 2, $k_1=\dots = k_r=r-1$ for a sharp example, 
since we need $M_i= \frac{n^2}{2(r-1)}-\frac{n}{2}$ for each $i$. But if $k_i<r-1$ for some $i$, then $M_i\geq \frac{n^2}{2 k_i}-\frac{n}{2} > \frac{n^2}{2(r-1)}-\frac{n}{2}$.

Hence a sharp example is necessarily in the intersection of Case 1 and Case 2,
and the bounds in both cases are sharp for it. 

We claim that the intersection of any two components of different colors must
have cardinality exactly $\frac{n}{(r-1)^2}$ (and consequently,
the cardinality of any monochromatic component is exactly
$\frac{n}{r-1}$). Let $i,j\in [r]$ be two different
colors. We know $k_i=k_j=r-1$ and by  (\ref{eq:komp_kul})
\begin{equation} 
\sum_{C_i\in \mathcal{C}_i,\; C_j\in \mathcal{C}_j} |C_i-C_j|=(r-2)\cdot n.
\end{equation}
Choose $C'_i\in \mathcal{C}_i$ and $C'_j\in \mathcal{C}_j$ such that
$s=|C'_i-C'_j|$ is minimum and recall from the proof of Case 1 that in this
case $C'_i\cap C'_j\ne\emptyset$. If $s<\frac{r-2}{(r-1)^2}n$, then for any
$x\in C'_i\cap C'_j$, the components $\cC(x,[r]-\{i\})$ cover each
vertex outside $C'_i-C'_j$, hence strictly more than
$\big(1-\frac{r-2}{(r-1)^2}\big)\cdot n$ vertices but this contradicts the
assumption.  Since $s$ is the minimum, it cannot be bigger than
the average, thus for any $C_i\in \mathcal{C}_i$ and $C_j\in \mathcal{C}_j$ we
have $|C_i-C_j|=\frac{r-2}{(r-1)^2}n$. Now take any $C_i\in \mathcal{C}_i$ and
$C_j,C'_j\in \mathcal{C}_j$. As $|C_i-C_j|=|C_i-C'_j|$, we also have $|C_i\cap
C_j|=|C_i\cap C'_j|$.  By symmetry we also get $|C_i\cap C_j|=|C'_i\cap C_j|$
for any $C_j\in \mathcal{C}_j$ and $C_i,C'_i\in \mathcal{C}_i$, proving the
claim.

Moreover, from $t= \frac{1}{r-1}\cdot \big[\sum_{i=1}^r M_i-\binom{n}{2}\big]$,  
for each edge $uv\in E(G)$ either $|\col(uv)|=1$ or $|\col(uv)|=r$.
 From this, the following useful property follows: 

\begin{claim}\label{cl:inters_nice}
If $C_1\cap \dots \cap C_r\neq \emptyset$ where $C_1\in \mathcal{C}_1, \dots,
C_r\in\mathcal{C}_r$, then for arbitrary $1\leq i< j\leq r$,
we have $C_i\cap C_j=C_1\cap \dots \cap C_r$.
\end{claim}
\begin{proof}
If there were a vertex $x\in C_1\cap \dots \cap C_r$ and a vertex $y\in C_i\cap C_j- C_\ell$ for some $\ell$, then the edge $xy$ would have color $i$ and $j$ but not color $\ell$, which would contradict the fact that either $|\col(xy)|=1$ or $|\col(xy)|=r$.
\end{proof}

Now let us take the following incidence structure $\mathcal{A}$: Let the points of $\mathcal{A}$ be the nonempty intersections $C_1\cap \dots \cap C_r\neq \emptyset$, where $C_1\in \mathcal{C}_1, \dots, C_r\in\mathcal{C}_r$. 
Let the lines of $\mathcal{A}$ be the monochromatic components of $G$. Let a point corresponding to $C_1\cap \dots \cap C_r\neq \emptyset$ be incident with the lines corresponding to $C_1, \dots , C_r$.
Since each vertex of $G$ is incident with edges of each color, this way each vertex of $G$ is mapped to a point of $\mathcal{A}$. Also, for a nonempty intersection, $C_1\cap \dots \cap C_r=C_1\cap C_2$. Since $|C_1\cap C_2|=\frac{n}{(r-1)^2}$, each point of $\mathcal{A}$ corresponds exactly to $\frac{n}{(r-1)^2}=:b$ vertices of $G$. 

We claim that $\mathcal{A}$ is an affine plane of order $r-1$. Moreover, we claim that two lines are disjoint if and only if the corresponding monochromatic components have the same color. Note that if we prove these statements, it follows that $G$ is the blowup of an affine plane.

We have already proved that two components of $G$ of different colors have a nonempty intersection. On the other hand, two monochromatic components of the same color are disjoint by the definition of a component. Hence indeed two lines in $\mathcal{A}$ are disjoint if and only if the corresponding monochromatic components have the same color.
To prove that $\mathcal{A}$ is an affine plane of order $r-1$, we need to check the five conditions given in Definition \ref{def:aff_plane}.

(i) %We need to show that any two vertices in $G'$ occur in exactly one common monochromatic component. 
%Any two points in $\mathcal{A}$ have a line incident to both of them, since taking a vertex from the intersections corresponding to the two points, these are connected by at least one edge (since $G$ is a complete graph), and this edge has some color.
We claim that the points corresponding to $C_1\cap \dots \cap C_r\neq \emptyset$ and $C'_1\cap \dots \cap C'_r\neq \emptyset$ where $C_1,C'_1\in \mathcal{C}_1, \dots, C_r,C'_r\in\mathcal{C}_r$ have at least one common monochromatic component. Indeed, take $x\in C_1\cap \dots \cap C_r$ and $y\in C'_1\cap \dots \cap C'_r$. Since $G$ is complete, $xy\in E(G)$. This edge has at least one color, hence $x$ and $y$ have a common monochromatic component.

Now we claim that these two points have at most one common monochromatic component. Indeed, by Claim \ref{cl:inters_nice}, if $C_i=C'_i$ and $C_j=C'_j$ for some $i\neq j$, then $C_1\cap \dots \cap C_r=C_i\cap C_j=C'_i\cap C'_j=C'_1\cap \dots \cap C'_r$.

(ii) Let $C$ be the monochromatic component of $G$ corresponding to the line $L$.
As we noted before, two monochromatic components in $G$ are disjoint if and only if they have the same color. Suppose that $C$ has color $i$. Let $C'$ be the component of color $i$ that contains $x$. The line corresponding to $C'$ satisfies the requirements of (ii). 

(iii) %We claim that each line contains $r-1$ points. Indeed, a line $L$ corresponding to the monochromatic component $C\in \mathcal{C}_i$ contains the vertices corresponding to the nonempty intersections of the form 
%$C_1\cap \dots \cap C_{i-1}\cap C \cap C_{i+1} \cap \dots \cap C_r$ where $C_j\in \mathcal{C}_j$ for $j=1,\dots ,i-1, i+1, \dots ,r$. From Claim \ref{cl:inters_nice}, these nonempty intersections coincide with the intersections of $C$ 
%with lines of color $j$ (for any fixed $j$). As each of the $r-1$ lines of color $j$ intersect $C$, this means $r-1$ intersections.
If there is a line containing only one point, let the monochromatic component of $G$ corresponding to the line be $C_i\in \mathcal{C}_i$ and the intersection corresponding to the point be $C_1\cap \dots \cap C_r\neq \emptyset$ where 
$C_1\in \mathcal{C}_1,\dots , C_r\in \mathcal{C}_r$. From the fact that the line has only one point, $C_i\subseteq C_1\cap \dots \cap C_{i-1}\cap C_{i+1}\cap \dots \cap C_r$.
But then $C_1,\dots C_{i-1},C_{i+1},\dots ,C_r$ cover all the vertices of $G$ since $G$ is complete. Thus, the example is not sharp.

(iv) %Each point is incident with exactly $r$ lines, since the points correspond to the nonempty intersections of type $C_1\cap \dots \cap C_r$ where $C_j\in \mathcal{C}_j$. As we get the points as the intersection of $r$ components, they are incident with $r$ lines.
It can be seen from the definition that each point of $\mathcal{A}$ is incident with $r\geq 3$ lines.

(v) This follows from the fact that two lines are parallel if and only if they correspond to monochromatic components of the same color, and for each color, there are exactly $r-1$ monochromatic components.

With this, we have proved that any sharp example needs to be a blowup of an affine plane. 
Now we prove that the blowup of an affine plane is always a sharp example.

We claim that $r-1$ monochromatic components of pairwise different colors cover at most $\big( 1-\frac{r-2}{(r-1)^2}\big)\cdot n= \big( 1-\frac{r-2}{(r-1)^2}\big) \cdot b(r-1)^2=b\cdot((r-1)^2-r+2)$ vertices.
Indeed, take the first component. This covers $b(r-1)$ vertices. The second component has a different color from the first, hence they have an intersection of size $\frac{n}{(r-1)^2}=b$. Hence the two components together cover at most $b(2(r-1)-1)$ vertices. And so on, each subsequent component needs to have an intersection of size at least $b$ with the union of the previous ones, hence altogether, they cover at most $b((r-1)^2-r+2)$ vertices. We can also see, that for covering exactly $b((r-1)^2-r+2)$ vertices, we need to take $r-1$ monochromatic components having a common intersection of $b$ points.
\end{proof}

\begin{remark}
In the case if $\big(1-\frac{r-2}{(r-1)^2}\big)\cdot n$ is not an integer, it would be reasonable to call the multi $r$-edge-colored $G$ sharp if the number of vertices coverable by $r-1$ monochromatic components of pairwise different colors is the minimum possible, i.e., 
$$\Bigl\lceil \big(1-\frac{r-2}{(r-1)^2}\big)\cdot n\Bigr\rceil.$$ 
We do not know the structure of the sharp examples in this sense.
\end{remark}

Recall the definition of a truncated projective plane:
\begin{defn}
	Take a projective plane of order $r-1$. The \emph{truncated projective plane of order $r-1$} is the following hypergraph: Remove a point and the lines incident to it from the projective plane. Let the vertices of the hypergraph be the remaining points, and the hyperedges be the remaining lines.
\end{defn}

Note that this is an $r$-partite $r$-uniform hypergraph (the partite classes correspond to the unremoved points that were contained by a removed line).
Truncated projective planes play an important role in the study of Ryser's
conjecture. They give a family of sharp examples. Moreover, the only other known family of extremals \cite{ABPSz} is also constructed using truncated projective planes.
In \cite{HNSz2} it is shown that the truncated
Fano-plane is the main building block in the characterization of the sharp hypergraphs for Ryser's conjecture in the case $r=3$.  In addition,
the near-extremal family recently constructed by Haxell and Scott \cite{HS} is also
based on truncated projective planes.

Note that if one switches the role of vertices and hyperedges, an affine plane becomes a truncated projective plane.
Hence Theorem \ref{thm:cover_diff_sharp} gives the following result for hypergraphs:
\begin{theorem}\label{thm:cover_diff_hyp_sharp}
	Let $\cH$ be an $r$-partite $r$-uniform intersecting hypergraph.
	The maximum number of hyperedges coverable by a multi-colored set of
        size  $r-1$ equals to
	$\big(1-\frac{r-2}{(r-1)^2}\big)\cdot |E(\cH)|$ if and only if $\cH$
        can be obtained from a truncated projective plane by taking $b$
        parallel copies of each hyperedge for some fixed integer $b$.
\end{theorem}

\section{Ryser's conjecture in the case
  \texorpdfstring{$\Delta(\cH)=2$}{Delta(H) = 2}}

For $r=2$, Ryser's conjecture follows from K\H onig's theorem. In this section, we prove Ryser's conjecture for the very special case $\Delta(\cH)=2$ and $r\geq 3$. We note, that in this special case, the hypergraph does not even need to be r-partite for Ryser's bound to hold.

\begin{theorem} 
Let $\cH$ be an $r$-uniform hypergraph with $r\geq 3$ and $\Delta(\cH)=2$. Then $\tau(\cH)\leq (r-1)\cdot \nu(\cH)$.
\end{theorem}
\begin{proof}
Let the dual of a hypergraph $\cH$ be the following hypergraph $\cH^*$, with multiple hyperedges possible:

 $$V(\cH^*)= E(\cH)$$
 $$E(\cH^*)= \{\{ e\in E(\cH): e\ni v\}: v\in V(\cH)\} \quad \textrm{taken as a multiset}.$$

We have $\cH^{**}=\cH$, hence vertices of $\cH$ correspond exactly to hyperedges in $\mathcal{H^*}$ and hyperedges of $\cH$ correspond exactly to vertices in $\mathcal{H^*}$. %tehát kölcsönösen egyértelmu megfeleltetés van $\cH$ pontjai és $\cH^*$ élei, ill $\cH$ élei és $\cH^*$ pontjai között.

Note that a set of vertices $T\subseteq V(\cH)$ covers the hyperedges of $\cH$
if and only if the corresponding hyperedge set in $\mathcal{H^*}$ covers the
vertices of $\mathcal{H^*}$, so $\tau(\cH)=\varrho(\cH^*)$.

The degree of a vertex of $\mathcal{H^*}$ is the cardinality of the
corresponding hyperedge of $\cH$. Hence $\cH$ is $r$-uniform if and only if
$\mathcal{H^*}$ is $r$-regular, consequently $\Delta(\cH^*)=r$.
By definition, $\alpha'(\cH^*)=\nu(\cH)$.

If $\Delta(\cH)=2$, then $\mathcal{H^*}$ is a hypergraph with hyperedge cardinalities one or two, and the statement of the theorem is equivalent to $\varrho(\mathcal{H^*})\leq (\Delta(\mathcal{H^*})-1) \alpha'(\mathcal{H^*})$.

We can suppose that there are no hyperedges of cardinality one in $\mathcal{H^*}$. Indeed, if a hyperedge of cardinality one is contained by a hyperedge of cardinality two, then we can remove the hyperedge of cardinality one. This does not change the value of $\alpha'$, and $\Delta=\Delta(\cH^*)$ can only decrease. 
Moreover, the value of $\varrho$ can only increase by removing a hyperedge, since a covering hyperedge set of the modified hypergraph is also a covering hyperedge set in the original hypergraph.
Hence if the statement is true for the hypergraph after removing a hyperedge, then the statement is also true for the original hypergraph.

If a hyperedge of cardinality one is not contained by a hyperedge of cardinality two, then this hyperedge (or a parallel copy of it) needs to occur in each hyperedge cover. Hence leaving this vertex and the cardinality one hyperedges incident to it, $\varrho$ decreases by one. On the other hand, $\alpha'$ also decreases by one and $\Delta$ can only decrease. Hence if the statement is true to the modified hypergraph, it is also true for the original hypergraph.

%Hence it is enough to prove, that for a graph $G$, 
%$\varrho(G)\leq (\Delta(G)-1)\cdot \alpha(G)$, as for graphs
%$\alpha'(G)=\alpha(G)$. 
The following lemma proves the theorem if the cardinality two hyperedges form a graph which is not a cycle.

\begin{lemma}
    If $G$ is a graph which is not a cycle, then $\varrho(G)\leq (\Delta(G)-1)\cdot \alpha(G)$.
\end{lemma}
  \begin{proof}
We will denote by $G[X]$ the subgraph of $G$ induced by the vertex set $X$. For a set of vertices $U\subseteq V$, we denote by $\Gamma(U)$ the set of neighbors of $U$.

The statement is easily seen to be true for complete graphs with at least four vertices, hence we can suppose that $G$ is not complete. 

Let $n=|V(G)|$.
Since $G$ is not a cycle, using Brooks' theorem, $G$ is colorable by $\Delta(G)$ colors. As consequence, $\alpha(G)\geq \frac{n}{\Delta}$.

Take an independent vertex set $I\subseteq V(G)$ of maximum size, and take a maximum matching $M$ in $G[V(G)-I]$.
Let $X=V(M)$ and $Y=V-I-X$. Since $M$ is a maximum matching in $G[V(G)-I]$, it follows that $Y$ is an independent set.
Hence $G[Y\cup I]$ is a bipartite graph.

We show that in $G[Y\cup I]$ there is a matching covering $Y$.
Suppose for contradiction that the condition of Hall's theorem is not satisfied, i.e., $\exists U \subseteq Y$ such that $|\Gamma(U)| < |U|$. Then $(I- \Gamma(U)) \cup U$ is an independent set, whose size is greater than $|I|$, which contradicts the choice of $I$.

Now take the following set of edges: the edges of $M$, the edges of a matching covering $Y$ in $G[Y\cup I]$, and for each thus uncovered vertex in $I$, an edge covering it. This is an edge cover of $G$ of cardinality at most $|M|+|Y|+(|I|-|Y|)=|M|+|I|$. Thus $\varrho\leq |M|+|I|$.

We show that $|M|+|I|\leq (\Delta(G)-1)\alpha(G)$. Indeed, since $|X|\leq n-|I|= n-\alpha(G)\leq n(1-1/\Delta)$, we have $|M|\leq \lfloor \frac{n(1-1/\Delta)}{2}\rfloor=\lfloor (\Delta-1)\frac{n}{2\Delta}\rfloor \leq \lfloor \frac{(\Delta-1)\alpha}{2} \rfloor$. Thus $\varrho\leq |M|+|I|\leq \lfloor \frac{(\Delta-1)\alpha}{2} \rfloor + \alpha \leq (\Delta-1)\alpha$.
\end{proof}

Now the only remaining case is if the cardinality two hyperedges of $H^*$ form a cycle, that is, $H^*$ is a cycle with some additional cardinality one hyperedges. Suppose that the cycle has $l$ vertices, plus there are $k$ isolated vertices. Then the vertex set of $H^*$ can be covered by $\lceil \frac{l}{2} \rceil + k$ hyperedges, and $\alpha'(H^*)=\lfloor \frac{l}{2} \rfloor + k$. Since $r=\Delta(H^*)>2$, this means $\varrho(\mathcal{H^*})\leq (\Delta(\mathcal{H^*})-1) \alpha'(\mathcal{H^*})$. %On the other hand, if $r=2$, then from the fact that $H$ is bipartite, $\chi'(H^*)=2$, hence $l$ is even, thus 
%$\varrho(\mathcal{H^*})=(\Delta(\mathcal{H^*})-1) \alpha'(\mathcal{H^*}).$
\end{proof}

\end{document}